\documentclass[12pt]{amsart}

\usepackage{tikz}
\usetikzlibrary{arrows, shapes}
\usepackage{xcolor}
\usepackage[utf8]{inputenc}
\usepackage[all]{xy}
\usepackage{lscape}
\usepackage{enumerate}
\usepackage{amsmath}
\usepackage{amssymb}
\usepackage{amsthm}
\usepackage[percent]{overpic}
\usepackage[linktocpage=true]{hyperref}
\usepackage{youngtab}

\numberwithin{equation}{section}
\usepackage{cleveref}
\usepackage{tikz-3dplot}
\usetikzlibrary{3d}

\usepackage{accents}

\makeatletter
\def\l@subsection{\@tocline{2}{0pt}{2.5pc}{5pc}{}}
\makeatother

\input{abbr.tex}
\newtheorem{fact}{Fact}[section]
\newtheorem{lemma}[fact]{Lemma}
\newtheorem{conjecture}{Conjecture}

\newtheorem*{theorem*}{Theorem}
\newtheorem{defi}[fact]{Definition}
\newtheorem{exa}[fact]{Example}
\newtheorem{cla}[fact]{Claim}

\newtheorem{proposition}[fact]{Proposition}
\newtheorem{corollary}[fact]{Corollary}

\newenvironment{definition}{\begin{defi} \rm}{\end{defi}}

\theoremstyle{definition}
\newtheorem{observation}[fact]{Empirical observation}

\theoremstyle{remark}
\newtheorem{nremark}[fact]{Remark}
\newenvironment{remark}
{\par\pushQED{\qed}\nremark \small}
{\popQED\endnremark}

\theoremstyle{remark}

\newcommand{\dst}{\displaystyle}

\newcommand{\Co}{\ensuremath{\mathbb{C}}}

\newcommand{\TT}{\ensuremath{\mathbb{T}}}

\newcommand{\CC}{\ensuremath{\mathbb{C}}}

\def \s {\sigma}

\def \C {\mathbb{C}}

\newcommand{\ac}{\ensuremath{\mathcal{A}}}

\newcommand{\bv}{\ensuremath{\mathbf{b}}}
\newcommand{\av}{\ensuremath{\mathbf{a}}}
\newcommand{\gb}{\ensuremath{\mathbf{g}}}
\newcommand{\vb}{\ensuremath{\mathbf{v}}}
\newcommand{\eb}{\ensuremath{\mathbf{e}}}
\newcommand{\fb}{\ensuremath{\mathbf{f}}}
\newcommand{\ub}{\ensuremath{\mathbf{u}}}
\newcommand{\wb}{\ensuremath{\mathbf{w}}}

\newcommand{\xb}{\ensuremath{\mathbf{x}}}

\title{Triplets of Mutually Unbiased Bases}
\author{M\'ate Matolcsi, \'Akos K.\ Matszangosz, D\'aniel Varga, Mih\'aly Weiner}

\address{M\'at\'e Matolcsi, HUN-REN Alfr\'ed R\'enyi Institute of Mathematics, Re\'altanoda utca 13-15, 1053 Budapest, Hungary,
and Department of Analysis and Operations Research,
Institute of Mathematics,
Budapest University of Technology and Economics,
M\H uegyetem rkp. 3., H-1111 Budapest, Hungary
}
\email{matomate@renyi.hu}

\address{\'Akos K.\ Matszangosz, HUN-REN Alfr\'ed R\'enyi Institute of Mathematics, Re\'altanoda utca 13-15, 1053 Budapest, Hungary}
\email{matszangosz.akos@gmail.com}

\address{D\'aniel Varga, HUN-REN Alfr\'ed R\'enyi Institute of Mathematics, Re\'altanoda utca 13-15, 1053 Budapest, Hungary}
\email{daniel@renyi.hu}

\address{Mih\'aly Weiner, Department of Analysis and Operations Research,
Institute of Mathematics,
Budapest University of Technology and Economics,
M\H uegyetem rkp. 3., H-1111 Budapest, Hungary}
\email{elektrubadur@gmail.com}

\date{July 2025, MSC2020 05B20, 81P15}

\begin{document}

\begin{abstract}
We initiate a systematic study of triplets of mutually unbiased bases (MUBs). We show that each MUB-triplet in $\Co^d$ is characterized by a $d\times d\times d$ object that we call a {\it Hadamard cube}. We describe the basic properties of Hadamard cubes, and show how an MUB-triplet can be reconstructed from such a cube, up to unitary equivalence. 

We also present an algebraic identity which is conjectured to hold for all MUB-triplets in dimension 6. If true, it would imply the long-standing conjecture of Zauner that the maximum number of MUBs in dimension 6 is three. 
\end{abstract}

\maketitle

\section{Introduction}

The study of mutually unbiased bases (MUBs) was originally motivated in quantum information theory  by the physical property that if a system is prepared in an eigenstate of one of the bases, then a measurement conducted in any of the  other bases yields all possible outcomes with equal probability. Due to this property, MUBs  find several applications in quantum algorithms such as dense coding, teleportation, entanglement swapping, covariant cloning, and quantum state tomography (see \cite{durt} and \cite{mweig}  for comprehensive surveys on MUBs and their applications). 

\medskip

Recall that two orthonormal bases,
$X=(\eb_1,\ldots,\eb_d)$ and $Y=(\fb_1,\ldots,\fb_d)$ in $\CC^d$, are
called \emph{unbiased} if for every $j,k\in [d]\equiv \{1,\ldots d\}$,
$|\langle \eb_j,\fb_k \rangle|=\frac{1}{\sqrt{d}}$. In general, we will say that two unit vectors $\ub$ and $\vb$ in $\CC^d$ are \emph{unbiased} if $|\langle \ub,\vb \rangle|=\dst\frac{1}{\sqrt{d}}$. A collection
$X_1,\ldots X_m$ of orthonormal bases is said to be
\emph{(pairwise) mutually unbiased} if every two of them are
unbiased. In this note we will be interested in the case $m=2$ or $3$, i.e. pairs and triplets of MUBs.

\medskip

The maximal number of MUBs in $\Co^d$ is known to be $\le d+1$ for every $d$ (see e.g. \cite{BBRV,BBELTZ,WF}), and to be exactly $d+1$ if $d$ is a prime-power (see e.g. \cite{BBRV,Iv,KR,WF}).
The lowest dimension where the existence of a complete system of $d+1$ MUBs is not known is $d=6$, where the maximal number of MUBs is conjectured to be 3 by Zauner \cite{Za}. 

\medskip

In this note we initiate a systematic study of triplets of MUBs. To this end, let us first introduce the natural notion of unitary equivalence between systems of MUBs. In notation, throughout the text we use $[n]=\{1, 2, \dots, n\}$.

\begin{definition}\label{du}
Suppose $X_1=(\eb^{(1)}_1,\ldots,\eb^{(1)}_d), \dots, X_m=(\eb^{(m)}_1,\ldots,\eb^{(m)}_d)$
 and $Y_1=(\fb^{(1)}_1,\ldots,\fb^{(1)}_d)$, $\dots$, $Y_m=(\fb^{(m)}_1,\ldots,\fb^{(m)}_d)$ are two systems of mutually unbiased bases. Let $P^{(k)}_j=|\eb^{(k)}_j\rangle \langle \eb^{(k)}_j|$ be the orthogonal projection onto the 1-dimensional subspace spanned by $\eb^{(k)}_j$ and, similarly, $\tilde{P}^{(k)}_j=|\fb^{(k)}_j\rangle \langle \fb^{(k)}_j|$.  
 We say that the two systems of MUBs $(X_1, \dots, X_m)$ and $(Y_1, \dots, Y_m)$ are {\it directly unitary equivalent} if there exists a unitary operator $U$ on $\Co^d$ such that $UP^{(k)}_jU^\ast = \tilde{P}^{(k)}_j$ for all $k \in [m]$ and $ j\in [d]$. 
\end{definition}

Note that the condition $UP^{(k)}_jU^\ast = \tilde{P}^{(k)}_j$ is equivalent to saying that $U\eb_j^{(k)}=\lambda_{j,k} \fb_j^{(k)}$ for some complex number $\lambda_{j,k}$ of absolute value $1$, i.e. the unitary operator $U$ takes the basis vectors of the first system into the basis vectors of the second "up to some phase factors". The phrase "directly" refers to the fact that $U$ respects the order of the bases, and the order of the vectors. This simplifies the notation in several proofs. However, it is more natural to introduce a notion of unitary equivalence where permutation of indices are allowed. 

\begin{definition}\label{pu}
We say that the two systems of MUBs $(X_1, \dots, X_m)$, $(Y_1, \dots, Y_m)$ are {\it permutationally unitary equivalent} if there exists a unitary operator $U$ on $\CC^d$, and permutations $\pi\in S_m$, $\s_1, \dots, \s_m\in S_d$ such that $UP^{(k)}_jU^\ast = \tilde{P}^{(\pi(k))}_{\s_k(j)}$ for all $k \in [m], j\in [d]$. 
\end{definition}

\begin{remark}
For any orthonormal basis $X$ in $\CC^d$, let $\ac_X$ denote the subalgebra of operators on $\CC^d$ which are diagonal in the basis $X$. It is easy to see that two systems of MUBs, $X_1, \dots, X_m$
 and $Y_1, \dots Y_m$ are permutationally unitary equivalent if and only if there exists a permutation $\pi\in S_m$ and a unitary operator $U$ on $\CC^d$ such that $U\ac_{X_i}U^\ast =\ac_{Y_{\pi(i)}}$.  
\end{remark}

The rest of the paper is devoted to studying pairs and triplets of MUBs up to unitary equivalence. In Section \ref{sec2} we revisit the characterization of unitary equivalence of pairs of MUBs in terms of complex Hadamard matrices and Haagerup invariants. In Section \ref{sec3} we introduce the main new concept of the paper, the notion of a {\it Hadamard cube}, and show how such a cube determines a MUB-triplet up to unitary equivalence. In Section \ref{sec4} we present some numerical evidence about Hadamard cubes in dimension 6, and formulate a conjectured algebraic identity that can lead to the proof of Zauner's conjecture on the maximal number of MUBs being three in this case. It is worth noting the philosophy behind this approach. Thus far, all attempts in the literature have failed to prove the non-existence of a quadruple of MUBs in dimension 6. The main new idea here is to carry out such a proof in two steps. First, provide a complete characterization of triplets of MUBs (which are known to exist, and hence numerical searches can help by giving insight). Then, as a second step, prove that these (explicitly given) MUB-triplets cannot be extended with a fourth basis. 

In spirit, this approach is similar to the one where one attempts to give a complete characterization of complex Hadamard matrices (i.e.\ MUB pairs, instead of triplets). However, as testified by the literature (see \cite{generic}),  there are just "too many" complex Hadamard matrices in dimension $6$, and an explicit description seems hopeless. On the contrary, as we shall explain, numerical evidence suggests that all triplets have a certain, easily understandable structure.

\section{Pairs of MUBs}\label{sec2}

In this section we consider pairs of mutually unbiased bases. The results below are basically reformulations of well-known facts from the literature, and our only purpose in stating them explicitly is to motivate the discussion for MUB-triplets in later sections. 

\medskip

Let $X=\{\eb_1,\ldots,\eb_d\}$ and $Y=\{\fb_1,\ldots,\fb_d\}$ be a pair of mutually unbiased bases in $\CC^d$. Let $P_j=|\eb_j\rangle \langle \eb_j|$ and $Q_k=|\fb_k\rangle \langle \fb_k|$ denote the orthogonal projections onto the subspaces spanned by $\eb_j$ and $\fb_k$. 

\medskip

It is well-known that the  pair $X, Y$ gives rise to a {\it complex Hadamard matrix} $H_{X,Y}$ with matrix elements given by the formula
\begin{equation}\label{had}
h_{j,k}=\sqrt{d}\langle\eb_j,\fb_k \rangle \; \;  \;
(j,k\in [d]).    
\end{equation}
 That is, the columns (and hence also the rows) of the matrix $H_{X,Y}\in M_d(\CC)\equiv \CC^{d\times d}$ are pairwise orthogonal and all entries of $H_{X,Y}$ have modulus 1. 

\medskip

Two complex Hadamard matrices $H_1$ and $H_2$ are called {\it equivalent}, in notation $H_1\cong H_2$, if

\begin{equation}\label{eqhad}
H_1=P_1 D_1 H_2 D_2 P_2    
\end{equation}

with some unitary diagonal matrices $D_1, D_2$
and permutation matrices $P_1, P_2$. Although not explicitly stated in the literature, we will see below that this notion of equivalence corresponds to the notion of permutational unitary equivalence of pairs of MUBs, as given in Definition \ref{pu}. First, however, let us recall the Haagerup invariants $g^H_{j,k,l,r}=h_{j,k}\overline{h_{k,l}}h_{l,r}\overline{h_{r,j}}$ corresponding to a complex Hadamard matrix $H$.
These invariants are often used in the study of complex Hadamard matrices. Notice that for us, with $H$ being the Hadamard matrix associated to the MUB-pair $X,Y$, we have 
\begin{equation}\label{haag}
g^H_{j,k,l,r}=\langle \eb_j, \fb_k\rangle \langle \fb_k, \eb_l \rangle \langle \eb_l, \fb_r\rangle \langle \fb_r, \eb_j\rangle ={\rm Tr}(P_j Q_k P_l Q_r)
\end{equation}
and hence $g^H_{j,k,l,r}$ remains unchanged if each of the basis vectors is multiplied by some (not necessarily the same) complex number of unit absolute magnitude -- that is why we say that these values are "invariants". In other words, if $H,\tilde{H}$ are Hadamard matrices and there exist two diagonal matrices $D_1,D_2$ such that $\tilde{H}=D_1 H D_2$, then $g^H=g^{\tilde{H}}$. 
\medskip

With these notions at hand we can now formulate rigorous conditions for the unitary equivalence of pairs of MUBs.

\begin{proposition}\label{pairdu}
Suppose $X=(\eb_1,\ldots,\eb_d)$, $Y=(\fb_1,\ldots,\fb_d)$ and $V=(\vb_1,\ldots,\vb_d)$, $W=(\wb_1,\ldots,\wb_d)$ are two pairs of mutually unbiased bases in $\CC^d$, with the corresponding  complex Hadamard matrices $H=H_{X,Y}$ and $\tilde{H}=H_{V,W}$ defined as in \eqref{had}. The following are equivalent: 
\begin{enumerate}
    \item the pair $(X, Y)$ is directly unitary equivalent (as in Definition \ref{du}) to the pair $(V, W)$,
    \item there exist two diagonal matrices $D_1,D_2$ such that $\tilde{H}=D_1 H D_2$,
    \item the corresponding Haagerup invariants are equal, i.e. $\forall j,k,l,r:\; g^H_{j,k,l,r}=g^{\tilde{H}}_{j,k,l,r}$,
    \end{enumerate}
\end{proposition}

Note here that we require $g^H_{j,k,l,r}=g^{\tilde{H}}_{j,k,l,r}$, for all quadruples  $j,k,l,r$, which is much stronger than just requiring the sets $\{ g^H_{j,k,l,r}: \ 1\le j,k,l,r \le d\}$ and $\{ g^{\tilde{H}}_{j,k,l,r}: \ 1\le j,k,l,r \le d\}$ being equal. The latter is not sufficient for equivalence of complex Hadamard matrices. 

\begin{proof}
The directions $(1)\Rightarrow(2)\Rightarrow (3)$ are trivial. We will show here only the implication $(3)\Rightarrow (1)$. 

Assume the Haagerup invariants corresponding to $H$ and $\tilde{H}$ are equal; $g^H=g^{\tilde{H}}=:g$. Multiplying each of the basis vectors with a "phase factor" -- i.e. with a complex number of unit length -- does not change the rank one projections determined by these vectors and, ultimately, whether these systems are directly unitary equivalent or not. It does not change the value of the Haagerup invariants either. Hence, by adjusting the phases of the vectors in $Y$, $\fb_j'=\alpha_j\fb_j$ ($|\alpha_j|=1)$,  we may arrange that  $\langle \fb_j', \eb_1\rangle =1/\sqrt{d}$. We may then adjust the phases of the vectors in $X$,  $\eb_k'=\beta_k\eb_k$ ($|\beta_k|=1)$, in such a way that $\langle \eb_k', \fb_1' \rangle=1/\sqrt{d}$. By a similar modification, we can achieve that $\langle \vb_j',\wb_k' \rangle = 1/\sqrt{d}$, whenever $j$ or $k$ is equal to 1. 

Then
$$
\left(\frac{1}{\sqrt{d}}\right)^3\langle\eb_j',\fb_k'\rangle =
\langle\eb_j',\fb_k'\rangle
\langle\fb_k',\eb_1'\rangle\, 
\langle\eb_1',\fb_1'\rangle\, 
\langle\fb_1',\eb_j' \rangle\, 
=
g_{j,k,1,1}
$$
and similarly, $(1/\sqrt{d})^3 \langle\vb_j',\wb_k'\rangle =g_{j,k,1,1}$ showing that $\langle\eb_j',\fb_k'\rangle =\langle\vb_j',\wb_k'\rangle$ for every $1\le j, k\le d$. 

It follows that for the uniquely determined unitary operator $U$ such that $U\eb_k' =\vb_k'$ (for all $k\in [d]$), we also have $U\fb_k' =\wb_k'$ (for all $k\in [d]$), and hence that $U$ establishes a direct unitary equivalence between the pairs $(X,Y)$ and $(V,W)$. 
\end{proof}

We can now connect the notion of equivalence of complex Hadamard matrices to the notion of permutational unitary equivalence of MUB-pairs. 

\begin{proposition}\label{pairpu}
Let $X=\{\eb_1,\ldots,\eb_d\}$, $Y=\{\fb_1,\ldots,\fb_d\}$ and $V=\{\vb_1,\ldots,\vb_d\}$, $W=\{\wb_1,\ldots,\wb_d\}$ be two pairs of mutually unbiased bases in $\CC^d$, with the corresponding  complex Hadamard matrices $H=H_{X,Y}$ and $\tilde{H}=H_{V,W}$ defined as in \eqref{had}.  
The following are equivalent: 
\begin{enumerate}
\item The complex Hadamard matrix $H$ is equivalent to $\tilde{H}$ or $\tilde{H}^\ast$. That is, $H=P_1 D_1 \tilde{H} D_2 P_2$ or $H=P_1 D_1 \tilde{H}^\ast D_2 P_2$  with some unitary diagonal matrices $D_1, D_2$
and permutation matrices $P_1, P_2$.

\item The pair $(X, Y)$ is permutationally unitary equivalent (as in Definition \ref{pu}) to the pair $(V, W)$
\end{enumerate}
\end{proposition}

\begin{proof}
The statement readily follows  from the previous proposition. All we need to observe is that the Hadamard matrices $H_{X,Y}$ and $H_{Y,X}$, corresponding to the pairs $(X,Y)$ and $(Y,X)$, are adjoints of each other, and a change of the order of vectors in $X$ (or $Y$) corresponds to multiplication of $H_{XY}$ by a permutation matrix from the left (or from the right).  
\end{proof}

\begin{remark}
In view of this proposition, it would make sense to use an alternative definition of equivalence of complex Hadamard matrices. Namely, it would be natural to call $H$ and $\tilde{H}$ equivalent if 
$H=P_1 D_1 \tilde{H} D_2 P_2$ or $H=P_1 D_1 \tilde{H}^\ast D_2 P_2$  with some unitary diagonal matrices $D_1, D_2$
and permutation matrices $P_1, P_2$. However, in the literature, the traditional definition of equivalence does not include the adjoint $\tilde{H}^\ast$.  
\end{remark}

\section{MUB-triplets and Hadamard cubes}\label{sec3}

In this section we consider triplets of mutually unbiased bases, and provide a characterization result up to unitary equivalence. To motivate the discussion below, observe that for any three vectors $\ub_1,\ub_2,\ub_3$ the product 

\begin{equation}
\label{product_comb_vect}
\langle \ub_1, \ub_2\rangle \,
\langle \ub_2, \ub_3\rangle \,
\langle \ub_3, \ub_1\rangle     
\end{equation}
is independent of the ``phases'' of the vectors; that is, replacing any of these  vectors $\ub$ by a vector $\ub'=\lambda \ub$, where 
$|\lambda|=1$ (henceforth written as $\ub'\sim \ub$) will cause no change in the product value. Furthermore, 
when all the vectors are of unit length, we may re-write (\ref{product_comb_vect}) as
\begin{equation}
\label{product_comb_vect_tr}
\langle \ub_1, \ub_2\rangle \,
\langle \ub_2, \ub_3\rangle \,
\langle \ub_3, \ub_1\rangle =
\rm{Tr}(\;
|\ub_1\rangle\!\langle \ub_1|  \ub_2\rangle \,
\langle \ub_2 | \ub_3\rangle \,
\langle \ub_3 |  
  = {\rm Tr}(P_1P_2P_3),
\end{equation}
where $P_j$ are the orthogonal projections  
onto the one-dimensional subspaces spanned by the vectors $\ub_j$
($j=1,2,3$).

\medskip

Motivated by the formula above, given a triplet of mutually unbiased bases $X=(\eb_1, \dots, \eb_d), Y=(\fb_1, \dots, \fb_d), Z=(\gb_1, \dots, \gb_d)$ in $\Co^d$, we will consider the normalized products
\begin{equation}\label{hcube1}
C_{j,k,l}= \sqrt{d^3}\,
\langle \eb_j, \fb_k\rangle 
\langle \fb_k, \gb_l\rangle
\langle \gb_l, \eb_j\rangle \;\;\;\; (1\le j,k,l\le d).
\end{equation}

We may think of $C$ as a cube of size $d\times d\times d$, filled with complex numbers of modulus 1. We shall say that it is the  {\it Hadamard cube associated with the MUB-triplet} $(X, Y, Z)$. Note that we may rewrite the defining equation above as
$$
C_{j,k,l}= h_{j,k} \,\tilde{h}_{k,l}
\,\tilde{\tilde{h}}_{l,j}
$$
where
$$
h_{j,k} = {\sqrt{d}}
\langle \eb_j, \fb_k\rangle,\;\;
\tilde{h}_{k,l} =\sqrt{d}
\langle \fb_k, \gb_l\rangle,\;\;
\tilde{\tilde{h}}_{k,l} ={\sqrt{d}}
\langle \gb_k, \eb_l\rangle
$$
are the matrix entries of the $3$ complex Hadamard matrices $H,\tilde{H}$ and $\tilde{\tilde{H}}$ corresponding to the MUB-pairs $(X, Y), (Y, Z)$ and $(Z, X)$. 

\medskip

We now summarize the basic properties of the Hadamard cube $C$. 

\begin{proposition}\label{hcube}
For the Hadamard cube $C$ associated with a MUB-triplet $(X, Y, Z)$, as defined in \eqref{hcube1}, we have that
\begin{itemize}
    \item[i)] its entries are normalized: $\forall j,k,l\in[d]$: $|C_{j,k,l}|=1$,
    \item[ii)] each of its two-dimensional cross-sections ("slices") form a complex Hadamard matrix, i.e. for any fixed $j\in [d]$ the slices $C_{j,\cdot,\cdot}, C_{\cdot,j,\cdot}$ and $C_{\cdot,\cdot,j}$ are complex Hadamard matrices, 
    \item[iii)] parallel slices are phase-equivalent complex 
    Hadamard matrices; that is, for the parallel cross-sections $C_{j,\cdot,\cdot}$ and 
    $C_{j',\cdot,\cdot}$,  there exist scalars $\lambda_k, \mu_l$ ($k, l\in [d]$) such that $C_{j',k,l}=\lambda_k \, C_{j,k,l}\,\mu_l$, and similarly for parallel cross-sections in the other two directions,

    \item[iv)] its one-dimensional cross-sections ("piercings") sum to $\sqrt{d}$, that is, for any fixed $k,l\in[d]$ we have $\sum_{j=1}^d C_{j,k,l}=\sqrt{d}$ , and similarly for piercings of the other two directions.
\end{itemize}
\end{proposition}

\begin{proof}

Using the fact that $X,Y, Z$ form an MUB-triplet, properties (i), (ii) and (iii) follow immediately from the defining equation \eqref{hcube1}.  

To prove (iv), we use the notation $P_j=|\eb_j\rangle \langle \eb_j|$, $Q_k=|\fb_k\rangle \langle \fb_k|$, $R_l=|\gb_l\rangle \langle \gb_l|$, and equation \eqref{product_comb_vect_tr}, to obtain
$\sum_{j=1}^d C_{j,k,l}=\sqrt{d^3}
\sum_{j=1}^d \rm{Tr} \ P_jQ_kR_l=\sqrt{d^3}
\rm{Tr} \ (\sum_{j=1}^d P_j)Q_kR_l=\sqrt{d^3}\rm{Tr} \ Q_kR_l=\sqrt{d}$, where we have used the fact the $\sum_{j=1}^d P_j=I$, and $\rm{Tr} \ Q_kR_l=1/d$. 
\end{proof}

As the properties (i)-(iv) are fundamental, it is natural to introduce the following definition. 

\begin{definition}
A cube of size $d\times d\times d$ is called a \emph{Hadamard cube} if it satisfies properties (i)-(iv) in Proposition \ref{hcube}.
\end{definition}

Note that -- for reasons of clarity -- we have listed property (i) separately, although it is redundant, as it follows from (ii) by definition. In fact, later in this section we shall see that there are further redundancies, and some weaker assumptions also imply conditions (i)-(iv). 

\medskip

Having introduced the concept of a Hadamard cube, we have two immediate questions to consider. First, given a Hadamard cube $C$, does there exist a MUB-triplet $(X, Y, Z)$ such that its associated Hadamard cube is exactly $C$? If so, does the Hadamard cube $C$ determine the  MUB-triplet $(X, Y, Z)$ up to unitary equivalence? Below we shall answer both questions in the affirmative.

\begin{proposition}
\label{prop:triplet_equiv}
Assume
$((\eb_j)_{j=1}^d,(\fb_j)_{j=1}^d,(\gb_j)_{j=1}^d)$ and
$((\tilde{\eb}_j)_{j=1}^d,(\tilde{\fb}_j)_{j=1}^d,(\tilde{\gb}_j)_{j=1}^d)$
are two MUB-triplets in $\C^d$. They
give rise to the same Hadamard cube, i.e. 
$$
\langle \eb_j, \fb_k\rangle 
\langle \fb_k, \gb_l\rangle
\langle \gb_l, \eb_j\rangle 
= 
\langle \tilde{\eb}_j, \tilde{\fb}_k\rangle 
\langle \tilde{\fb}_k, \tilde{\gb}_l\rangle
\langle \tilde{\gb}_l, \tilde{\eb}_j\rangle 
$$
for all $j,k,l\in [n]$, if and only if the MUB-triplets  are directly unitary equivalent in the sense of Definition \ref{du}.
\end{proposition}
\begin{proof}
One direction is trivial: if the MUB-triplets are directly unitary equivalent, they give rise to the same Hadamard cube. 

\medskip

For the other direction, assume that the cubes $C$ and $\tilde{C}$ determined by the MUB-triplets in question are the same. By fixing $l=1$, we obtain
$$
\langle \eb_j, \fb_k\rangle 
\langle \fb_k, \gb_1\rangle
\langle \gb_1, \eb_j\rangle 
= 
\langle \tilde{\eb}_j, \tilde{\fb}_k\rangle 
\langle \tilde{\fb}_k, \tilde{\gb}_1\rangle
\langle \tilde{\gb}_1, \tilde{\eb}_j\rangle \ \ \ \ \ \forall j,k\in [d],
$$
showing that $D_1HD_2 = \tilde{D}_1\tilde{H}\tilde{D}_2$ where $H,\tilde{H}$ are the Hadamard matrices associated to the pairs $((\eb_j)_{j=1}^d,(\fb_j)_{j=1}^d$ and
$((\tilde{\eb}_j)_{j=1}^d,(\tilde{\fb}_j)_{j=1})$ respectively, and $D_1,D_2,\tilde{D}_1,\tilde{D}_2$
are the diagonal matrices with diagonal entries 
$(\langle \gb_1, \eb_j\rangle)_{j=1}^{d}$,
$(\langle \fb_k, \gb_1\rangle)_{k=1}^{d}$,
$(\langle \tilde{\gb}_1, \tilde{\eb}_j\rangle)_{j=1}^{d}$,
and 
$(\langle \tilde{\fb}_k, \tilde{\gb}_1\rangle)_{k=1}^{d}$, respectively) Thus, by proposition \ref{pairdu}, there exists a unitary $U$ which makes the pairs $((\eb_j)_{j=1}^d,(\fb_j)_{j=1}^d)$ and
$((\tilde{\eb}_j)_{j=1}^d,(\tilde{\fb}_j)_{j=1})$ directly unitary equivalent. We only need to prove that $U$ also connects the third elements of the two triplets; i.e.\ that for the projections $R_l=|\gb_l\rangle\langle\gb_l|$ and 
$\tilde{R}_l=|\tilde{\gb}_l\rangle\langle\tilde{\gb}_j|$ we have 
$UR_lU^* =\tilde{R}_l$,
for all $j\in [d]$.

Recall that the set of matrices $M_d(\CC)$ forms an inner product space with the {\it Hilbert-Schmidt} inner product $\langle A,B\rangle_{H}\equiv {\rm Tr}(A^*B)$ and consider the $d^2$ matrices 

\begin{equation}\label{pq}
A_{j,k} := | \eb_j\rangle\langle\eb_j|
\fb_k\rangle\langle\fb_k|=P_jQ_k\;\;\;\;
(j,k\in [d]).    
\end{equation}
Using the MUB-property, direct computation shows that 
\begin{equation}
\label{eq:A_jk}
\langle A_{j,k},A_{j',k'}\rangle_{H}
 = \delta_{j,j'}\delta_{k,k'}
 \frac{1}{d}\;\;\;\; (j,j',k,k'\in [d]),    
\end{equation}
i.e.\ that these $d^2$ matrices form an orthogonal basis of $M_d(\CC)$, normalized to $1/\sqrt{d}$. Clearly, the same
is true regarding the (possibly different) set of $d^2$ matrices $$\tilde{A}_{j,k}: = UA_{j,k}U^* = | \tilde{\eb}_j\rangle\langle\tilde{\eb}_j|
\tilde{\fb}_k\rangle\langle\tilde{\fb}_k|\;\;\;\; (j,j',k,k'\in [d]).$$  
Then for every $j,k,l\in [d]$,
\begin{eqnarray}
\nonumber
\langle UR_lU^*,\tilde{A}_{j,k}\rangle_{H}
&=& \langle UR_lU^*,UA_{j,k}U^*\rangle_{H}
=\langle R_l,A_{j,k}\rangle_{H} \\
&=&{\rm Tr}(R_lA_{j,k})=
\langle \gb_l, \eb_j\rangle 
\langle \eb_j, \fb_k\rangle 
\langle \fb_k, \gb_l\rangle
=\frac{1}{\sqrt{d^3}}C_{j,k,l}.
\end{eqnarray}
We similarly  obtain that 
$$
\langle \tilde{R}_l,\tilde{A}_{j,k}\rangle_{H} ={\rm Tr}(\tilde{R}_l\tilde{A}_{j,k})=
\langle \tilde{\gb}_l, \tilde{\eb}_j\rangle 
\langle \tilde{\eb}_j, \tilde{\fb}_k\rangle 
\langle \tilde{\fb}_k, \tilde{\gb}_l\rangle
=\frac{1}{\sqrt{d^3}}\tilde{C}_{j,k,l}=\frac{1}{\sqrt{d^3}}C_{j,k,l};
$$
i.e.\ $UR_lU^*$ and $\tilde{R}_l$
have the same inner product with all elements of the basis $(\tilde{A}_{j,k})_{(j,k\in[d])}$,
and hence $UR_lU^*= \tilde{R}_l$.
\end{proof}

Next we investigate the redundancies in the properties (i)-(iv) of Proposition \ref{hcube}. Namely, we show that a set of weaker assumptions on a cube $C$ is sufficient to imply that all of the properties (i)-(iv) hold. In the same proof we also show that properties (i)-(iv)  guarantee that $C$ is associated with an MUB-triplet. 

\medskip

Before we proceed to our formal proposition, we make some remarks on the defining properties of Hadamard cubes. First, note that given any three complex Hadamard matrices $H,\tilde{H}$ and $\tilde{\tilde{H}}$, the product 
\begin{equation}\label{gencube}
C_{j,k,l}:= H_{j,k} \tilde{H}_{k,l}
\tilde{\tilde{H}}_{l,j}
\end{equation}
satisfies  properties (i)-(iii) of Proposition \ref{hcube} automatically, while property (iv) is not necessarily satisfied. This shows that the last property does not follow from the previous ones. It is also quite clear that neither property (ii), nor property (iii) is a consequence of the other ones. So instead of ``dropping'' either of them, 
we shall consider a weaker form of 
each. For example, by property (ii), any pair of one-dimensional piercings contained in a 2 dimensional slice, as elements of $\mathbb C^d$, are orthogonal. However, it is considerably weaker to ask for parallel 2-dimensional slices, as elements of $\mathbb C^{d^2}$, to be orthogonal.
Also, by $(iii)$ it 
follows that if we consider a 
sub-cuboid of our Hadamard cube and mark its vertices with black and white in a checkered manner, then the product of the four ``black'' entries will coincide the with the product of its four ``white'' entries. That is, for any triplets of indices $j,k,l$ and $j',k',l'$
we have the following {\it Haagerup condition}
\begin{equation}
\label{black_and_white}
C_{j',k,l} C_{j,k',l} C_{j,k,l'} C_{j',k',l'} = C_{j,k,l} C_{j',k',l} C_{j,k',l'} C_{j',k,l'}.
\end{equation}
After this preparation we can describe a minimal set of requirements for a cube to be a Hadamard cube. 

\begin{proposition}\label{misi}
Assume $C=(C_{j,k,l})_{j,kl,\in [d]}$ is an array of complex numbers satisfying the following conditions:
\begin{itemize}
 \item[i')] its entries are normalized: $\forall j,k,l\in[d]$: $|C_{j,k,l}|=1$,
\item[ii')] at least one of its faces -- which we will call the ``bottom face'' -- is a complex Hadamard matrix, and the $d$ slices parallel to the bottom face are pairwise orthogonal, as elements of $\mathbb C^{d^2}$, 
\item[iii')] for any triplets of
indices $j,k,l$ and $j',k',l'$ equation (\ref{black_and_white}) holds,
\item[iv')]
entries in any one-dimensional ``horizontal'' cross-section (i.e.\ one parallel to the bottom face) 
sum to $\sqrt{d}$.
\end{itemize}
Then $C$ is a Hadamard cube, i.e. it satisfies properties $(i-iv)$ of Proposition \ref{hcube}. Also, there exists an MUB-triplet $(X, Y, Z)$ such that $C$ is the Hadamard cube associated with $(X, Y, Z)$.
\end{proposition} 

\begin{proof}
It is enough to show the last statement, i.e. that there exists an MUB-triplet associated to the cube $C$. Properties (i)-(iv) of Proposition \ref{hcube} will then follow automatically. 

\medskip

We can pick an arbitrary orthonormal basis $(\eb_j)_{j\in [d]}$ to be the first element of our MUB-triplet. To fix notations, let us say that the ``bottom face'' of $C$ is the one corresponding to the indeces $\{(j,k,l)\in [d]^3\,|\, l=d\}$. We then use this face to  define the second term of our MUB triplet; let
$\fb_j:= \frac{1}{\sqrt{d}} \sum_{k=1}^d C_{k,j,d}\eb_k$. The fact that $C_{\cdot,\cdot,d}$ is a complex Hadamard matrix implies that 
$(\fb_j)_{j\in [d]}$ is an orthonormal basis. Moreover, we have that $\langle \eb_k,\fb_j\rangle =\frac{1}{\sqrt{d}} C_{k,j,d}$; in particular, $\left((\eb_j)_{j\in [d]},(\fb_j)_{j\in [d]}\right)$
is an MUB pair. The  problem is to find a suitable third basis to form an MUB-triplet.  

\medskip

Let us look at this problem not in term of basis {\it vectors}, but -- as we have done before -- in terms of rank one orthogonal projections. Let $P_k:=|\eb_k\rangle\langle\eb_k|, Q_k:= |\fb_k\rangle\langle\fb_k| $ ($k\in [d]$) and $A_{j,k}:=P_jQ_k$ $(j,k\in [d])$ as in the proof of Prop.\ \ref{prop:triplet_equiv}. Instead of the vectors of the third basis, -- whose ``phases'', in any case, are not determined by our cube -- we need to find $d$ orthogonal projections $R_k$ $(k\in [d])$ such that
\begin{itemize}
\item[(a)] ${\rm Tr}(R_k R_j)=\delta_{k,j}$; this ensures that they are all rank one and are pairwise orthogonal -- i.e.\ they correspond to an ONB,
\item[(b)] ${\rm Tr}(P_k R_j)={\rm Tr}(Q_k R_j)=1/d$ for all $k,j\in [d]$; to ensure that the three bases form a MUB-triplet, and
\item[(c)] ${\rm Tr}(P_k Q_jR_l)= 1/\sqrt{d^3} \, C_{j,k,l}$ for all $j,k,l\in [d]$; to ensure that $C$ is indeed the Hadamard-cube corresponding to the constructed MUB-triplet.
\end{itemize}

As was already noted before, the matrices $A_{j,k}=P_jQ_k$ $(j,k\in [d])$ defined in \eqref{pq} form an orthogonal basis of $M_d(\CC)$, normalized to $1/\sqrt{d}$. It follows that with the definition
$$
R_l:=\frac{1}{\sqrt{d}}\sum_{j,k \in [d]}C_{j,k,l}\, A_{j,k}^* ={\sqrt{d}} \sum_{j,k\in [d]}C_{j,k,l}\, Q_kP_j\;\;\;
(l\in[d])
$$
we have that
$${\rm Tr}(P_j Q_kR_l)= 
{\rm Tr}(A_{j,k}R_l)=
\langle A_{j,k}^*,R_l\rangle_{H} = \frac{1}{\sqrt{d^3}} \, C_{j,k,l},$$
as required by point $(c)$. Moreover, we have 
\begin{eqnarray}
\nonumber
{\rm Tr}(Q_k R_l) &=& \frac{1}{\sqrt{d}}
\sum_{j',k'\in [d]}
 C_{j',k',l}
{\rm Tr}(Q_k A^*_{j',k'})=
\frac{1}{\sqrt{d}}
\sum_{j',k'\in [d]}
 C_{j',k',l}
{\rm Tr}(Q_k Q_{k'}P_{j'})
\\
\nonumber
&=&
\frac{1}{\sqrt{d}}
\sum_{j'\in [d]}
C_{j',k,l}
{\rm Tr}(Q_k P_{j'}) =
\frac{1}{d\sqrt{d}}
\sum_{j'\in [d]}
C_{j',k,l} =\frac{1}{d},
\end{eqnarray}
where the last equality follows by assumption (iv'). Similar computation shows that 
${\rm Tr}(P_k R_l)$ is also equal to $1/d$. Therefore,  point $(b)$ above is also satisfied.

\medskip

Since we have the coefficients of $R_l$ $R_{l'}$ in an orthogonal basis of $M_d(\CC)$, it is easy 
to compute the Hilbert-Schmidt product:
$$
\langle R_l,R_{l'}\rangle_{H}=
\frac{1}{d^2}\sum_{j,k} 
\overline{C}_{j,k,l} C_{j,k,l'} =
\delta_{l,l'},
$$
where we have used assumptions (i') and (ii'). Thus, ${\rm Tr}(R_l^*R_{l'})=\delta_{l,l'}$; which is almost what point $(a)$
asks for. What remains to be shown is that each $R_l$ is in fact an orthogonal projection; i.e.\ that
$R_l^* = R_l=R_l^2$. In order to proceed, let us note that
\begin{eqnarray}
\nonumber
A_{j,k}A_{j',k'} &=& 
|\eb_j\rangle\langle \eb_j|
\fb_k\rangle\langle \fb_k|
\eb_{j'}\rangle\langle \eb_{j'}|
\fb_{k'}\rangle\langle \fb_{k'}| 
=
\frac{\langle \eb_j,
\fb_k\rangle\langle \fb_k,
\eb_{j'}\rangle\langle \eb_{j'},
\fb_{k'}\rangle}{\langle \eb_j,\fb_{k'}\rangle}
\,
|\eb_j\rangle\langle \eb_{j}
|\fb_{k'}\rangle\langle \fb_{k'}|
\\
\nonumber
&=&
\frac{1}{d}\,
\frac{C_{j,k,d} \overline{C_{j',k,d}} C_{j',k',d}}{ C_{j,k',d}}
\,
A_{j,k'}
=
\frac{1}{d}\,
\frac{C_{j,k,d} C_{j',k',d}}{ 
C_{j',k,d}
C_{j,k',d}}
\,
A_{j,k'},
\end{eqnarray}
where we have used that for unit complex numbers conjugation is equivalent to taking reciprocals.
In particular, 
\begin{equation}
\label{eq:TrAA}
{\rm Tr}(A_{j,k}A_{j',k'})=
\frac{1}{d}\,
\frac{C_{j,k,d} C_{j',k',d}}{ 
C_{j',k,d}
C_{j,k',d}}
{\rm Tr}(P_jQ_{k'}) =
\frac{1}{d^2}\,\frac{C_{j,k,d} C_{j',k',d}}{ 
C_{j',k,d}
C_{j,k',d}}.
\end{equation}

We can now show that $R_l$ is an orthogonal projection. We have that
\begin{eqnarray}
\nonumber
(R_l^*)^2 &=& \left(\frac{1}{\sqrt{d}}\sum_{j,k \in [d]}\overline{C_{j,k,l}}\, A_{j,k}\right)^2 = 
\frac{1}{d}
\sum_{j,k,j',k'\in [d]}
\overline{C_{j,k,l}}\,
\overline{C_{j',k',l}}
A_{j,k}A_{j',k'}
\\
\nonumber
&=&
\frac{1}{d^2}
\frac{C_{j,k,d} \, C_{j',k',d}}{ 
C_{j,k,l}
C_{j',k',l} C_{j,k',d} C_{j',k,d}}
A_{j,k'} = 
\frac{1}{d^2}
\frac{C_{j,k,d} \, C_{j',k',d}}{ 
C_{j',k,l}
C_{j,k',l} C_{j,k,d} C_{j',k',d}}
A_{j,k'}
\\
\nonumber
&=&
\frac{1}{d^2}
\sum_{j,k,j',k'\in [d]}
\overline{C_{j',k,l}}\,
\overline{C_{j,k',l}}
A_{j,k'} = 
\left(\frac{1}{\sqrt{d^3}}
\sum_{j',k\in [d]}
\overline{C_{j',k,l}} \right)
\left(\frac{1}{\sqrt{d}}
\sum_{j,k'\in [d]}
\overline{C_{j,k',l}}
A_{j,k'} \right)
\\
\nonumber
&=&
\left(\frac{1}{\sqrt{d^3}}
\sum_{k\in [d]}
\sqrt{d}
\right) R_l^*
=R_l^*,
\end{eqnarray}
where from the second line to the third we used assumption (iii') -- i.e. the relation (\ref{black_and_white}) -- to ``reshuffle'' indices, and assumption (iv') from the third line to the fourth. Thus, the only thing that remains to be proven is the self-adjointness
of $R_l$.

To show that $R_l$ and $R_l^*$
coincide, it is enough to check that their Hilbert-Schmidt product with each $A_{j,k}^*$ ($j,k\in [d]$) coincide since the latter operators form an orthogonal basis of $M_d(\CC)$. 
We already know that
$\langle A_{j,k}^*,R_l\rangle_{H} = {\rm Tr}(P_j Q_k R_l)=1/\sqrt{d^3} C_{j,k,l}$. Also,
\begin{eqnarray}
\nonumber
\langle A_{j,k}^*,R_l^*\rangle_{H} &=& \frac{1}{\sqrt{d}}
\sum_{j',k'\in [d]} \overline{C_{j',k',l}}\,
{\rm Tr}(A_{j,k} A_{j',k'})=
\frac{1}{d^2\sqrt{d}}
\sum_{j',k'\in [d]} 
\frac{C_{j,k,d}\, C_{j',k',d}}{ 
C_{j',k',l} C_{j',k,d} C_{j,k',d}}
\\
\nonumber
&=&
\frac{1}{d^2\sqrt{d}}
\sum_{j',k'\in [d]} 
C_{j,k,l}
\overline{C_{j',k,l}}\, \overline{C_{j,k',l}}
=\frac{1}{d^2\sqrt{d}}
C_{j,k,l}
\sum_{j'\in [d]} 
\overline{C_{j',k,l}}\, 
\sum_{k'\in [d]} 
\overline{C_{j,k',l}}
\\
\nonumber
&=&
\frac{1}{\sqrt{d^3}}
C_{j,k,l}
\end{eqnarray}
where again from the first line to the second we used relation
(\ref{black_and_white}), and from 
the second line to the third assumption (iv'). The proof is now complete. 
\end{proof}

To end this section we give a natural generalization of the notion of Hadamard cubes. Notice that the defining equations (i)-(iv) in Proposition~\ref{hcube1} include the operation of conjugation. This is so, because the condition $|C_{j,k,l}|=1$ can only be expressed as $C_{j,k,l}\overline{C_{j,k,l}}=1$. However, keeping in mind Hilbert's Nullstellensatz, it could be advantageous to express all properties of the cube as polynomials (or, at least, rational functions)  of its entries. Therefore, we introduce the notion of {\it inverse orthogonal cubes} as follows. 

\begin{definition}\label{io}
A $d\times d\times d$ array $C$ of complex numbers is called an {\it inverse orthogonal cube}, if the following are satisfied: 
\begin{itemize}
 \item[a)] All 2-dimensional slices of $C$ are inverse-orthogonal matrices, i.e. for any two distinct column (or row) $u, v$ of a slice we have $\sum_{j=1}^d u_j/v_j=0$,
\item[b)] for any choices of indies $j, k, l$ and $j', k', l'$ we have the Haagerup condition \eqref{black_and_white},
\item[c)] for all 1-dimensional cross-sections of $C$ sum to $\sqrt{d}$, that is, $\sum_{j=1}^d C_{j,k,l}=\sqrt{d}$ for any fixed $k,l\in[d]$, and similarly for piercings of the other two directions.
\end{itemize}
\end{definition}

\medskip

We hope to use these polynomial equations to deduce non-trivial conclusions about Hadamard cubes.  

\section{Hadamard cubes in dimension 6}\label{sec4}

In this section we turn to the special case $d=6$, the lowest dimension where the maximal number of MUBs is not known. 

\medskip

Fairly strong numerical evidence suggests that the maximal number of MUBs in dimension 6 is three. Instead of trying to prove the non-existence of a quadruple of MUBs, it is natural to try to characterize all triplets of MUBs. If we succeed in doing so, the non-existence result for quadruples may follow automatically as a consequence of previous results in the literature. Following this strategy, we will formulate some conjectures about $6\times 6\times 6$  Hadamard cubes based on numerical evidence, and outline how these conjectures could lead to the proof of Zauner's conjecture that the maximal number of MUBs in dimension 6 is three.

\medskip

We begin by recalling a few facts about complex Hadamard matrices in dimension 6. While a complete characterization is still missing, it is believed that complex Hadamard matrices in dimension 6 form a connected 4-parameter manifold, and a single isolated point outside this manifold \cite{generic}. An online catalogue of complex Hadamard matrices is maintained at \cite{cat}.

Let us recall some notable 2-parameter families of complex Hadamard matrices here. 

\medskip

The {\it transposed Fourier family} of complex Hadamard matrices of order 6 is the two-parameter family
\begin{equation}\label{eq:fourier}
F^T(x,y)=\left[
		\begin{array}{cccccc}
			1 & 1 & 1 & 1 & 1 & 1 \\
			1 & \om^2 & \om & x &  \om^2 x& \om x \\
			1 & \om & \om^2 & y &  \om y& \om^2y \\
			1 & 1 & 1 & -1 & -1 & -1 \\
			1 & \om^2 & \om & -x &  -\om^2x& -\om x \\
			1 &  \om& \om^2 & -y &  -\om y& -\om^2 y
		\end{array}\right]
\end{equation} where $\om=e^{2i\pi/3}$, and $x, y$ are arbitrary complex numbers of unit modulus.

\medskip

In standard terminology, it is more customary to flip the rows and columns  and  define the {\it Fourier family}  $F(x, y)$ as the transposed of \eqref{eq:fourier}. However, for the discussion below, it is more convenient for us to display the family $F^T(x,y)$ above.  

\medskip

We will also need the {\it Sz\"oll\H osi} family $X(\alpha)$ described in \cite{Sz}, whose elements can be represented in the form

\begin{equation}\label{eqsz}
S=\left[
	\begin{array}{cccccc}
		a & b & c & d & e & f \\
		c & a & b & f & d & e \\
		b & c & a & e & f & d \\
		\conj{d} & \conj{f} & \conj{e} & -\conj{a} & -\conj{c} & -\conj{b} \\
		\conj{e} & \conj{d} & \conj{f} & -\conj{b} &  -\conj{a}& -\conj{c} \\
		\conj{f} & \conj{e}& \conj{d} &  -\conj{c}&  -\conj{b}& -\conj{a}
	\end{array}\right]
\end{equation}

Sz\"oll\H osi \cite{Sz} showed that complex Hadamard matrices of this form constitute -- up to the equivalence relation \eqref{eqhad} -- a two-parameter family $X(\alpha)$, parametrized by a complex parameter $\alpha$. In the notation  \eqref{eqsz} we have suppressed the dependence on $\alpha$, as we will only need the fact that the matrices $S$ have the 2-circulant form above. 

\medskip

It is worth mentioning two known facts here. First, the family $X(\alpha)$ is self-adjoint in the sense that if $S\in X(\alpha)$ then $S^\ast\in X(\alpha)$ (up to equivalence). Second, the  families $F^T(x,y)$ and $X(\alpha)$ intersect each other in the matrix $F^T(\omega,\omega)$. 

\medskip

Let us also recall here the most general known construction of MUB-triplets in dimension 6, based on an idea of Zauner \cite{Za}, and described in full generality by Sz\"oll\H osi in \cite{Sz}. 

\medskip

Given any $6\times 6$ complex Hadamard matrix $H$ in a 2-circulant form 
\begin{equation}\label{hb}
H=
\begin{bmatrix}
    A      & B  \\
    C      & D
\end{bmatrix},
\end{equation}
where $A, B, C$ and $D$ are circulant matrices, $H$ can be written as a product
\begin{equation}\label{hff}
H=F_1^{-1}F_2,
\end{equation}
where $F_1=D_1 F(x_1,y_1)$ and $F_2=D_2F(x_2,y_2)$, with some unitary diagonal matrices $D_1, D_2$, and some members $F(x_1,y_1)$ and $F(x_2,y_2)$ of the Fourier family \eqref{eq:fourier}. As such, every matrix $S\in X(\alpha)$ of the form \eqref{eqsz} can be written as a product
\begin{equation}\label{eq:zauner}
S=F_1^{-1}F_2.
\end{equation} Therefore the matrices $(I, \frac{1}{\sqrt{6}}F_1, \frac{1}{\sqrt{6}}F_2)$ form a 2-parameter family of MUB-triplets (the dependence of $S, F_1$ and $F_2$ on the complex parameter $\alpha$ is suppressed in the notation).

\medskip

The Hadamard cube $C$, associated with such an MUB-triplet $(I, \frac{1}{\sqrt{6}}F_1, \frac{1}{\sqrt{6}}F_2)$, has a very special form. For the orientation of the cube let us assume that horizontal slices are equivalent to $S$. The $j$th horizontal slice, by definition, is given by 
\begin{equation}\label{eq:cjk}
	C_{j,k,l}=(I^* F_1)_{j,k} (F_1^*F_2)_{k,l} (F_2^* I)_{l,j}
\end{equation}

A direct calculation shows 
that each slice $C_{j, \cdot, \cdot}$ consists of four circulant blocks of size $3\times 3$. As such, each slice $C_{j, \cdot, \cdot}$ consists of at most 12 distinct values, and hence the cube $C$ contains at most 72 distinct values, 
Further straightforward calculations show that these 72 values can be paired into 36 conjugate pairs. After a permutation of indices we arrive at a block decomposition of the cube $C$ into mini-cubes of size $2\times 2\times 2$, where the entries at opposite corners of each mini-cube are conjugates of each other. This is shown in Figure \ref{fig_cube}. (Note that  any value inside the cube can be filled in by equation \eqref{black_and_white}.) We call such a  cube $C$ {\it generic} if the 36 conjugate pairs of numbers appearing in it are all distinct.

\medskip

The other example of a 2-circulant matrix in \ref{hb} is when 
$$
A=B=C=\begin{bmatrix}
    \omega      & 1  & 1\\
    1      & \omega  & 1 \\
    1     & 1 & \omega
\end{bmatrix},
$$
and $D=-A$. In that case, $H$ defined in \ref{hb} is equivalent to the standard Fourier matrix $F_6(1,1)$ (in the notation of \ref{eq:fourier}), and the matrices $F_1, F_2$ in the construction \ref{hff} are also equivalent to $F_6(1,1)$. Furthermore, direct calculation shows that the construction \ref{hff}  leads to an  MUB-triplet $(X_1, X_2, X_3)$ such that each of the three transition matrices $X_jX_k^\ast$ are equivalent to the matrix $\frac{1}{\sqrt{6}}F(1, 1)$, and the associated Hadamard cube $C$ consists of 24th roots of unity exclusively. We call the arising cube {\it exceptional}, and remark that it is unique up to  permutations of indices,
and applying conjugation to all elements. 

\medskip

We now turn to describing the results of some numerical experiments that we have conducted in connection with MUB-triplets in dimension 6.

\medskip

Following common methodology \cite{raynal, BBELTZ, BWB} we formulate the numerical search for MUB-triplets as a continuous optimization task. The loss function we use to quantify the mutually unbiased properties of a set of matrices is the simple unweighted sum of two terms: one quantifying orthogonality, the other quantifying unbiasedness:

$$\sum_{1 \leq i \leq 3} \sum_{1 \leq k,l \leq 6} |(X_i^* X_i - I)_{kl}|^2 + \sum_{1 \leq i < j \leq 3} \sum_{1 \leq k,l \leq 6} (|(X_i^* X_j)_{kl}| - \frac{1}{\sqrt 6})^2,$$

This loss function is nonnegative, and it is zero if and only if all the MUB properties are satisfied. We use gradient-based optimization to find minima of this loss function, starting the optimization from three random  unimodular matrices (scaled by a factor of $\frac{1}{\sqrt{6}}$). The loss function is non-convex, so the optimization process can get stuck at local minima.

\medskip

After repeating the above process appropriately, we have obtained a collection of $20000$ numerical MUB-triplets $(X_1, X_2, X_3)$.
We summarize the results of this numerical experiment below. 

\begin{figure}
  \centering

\input{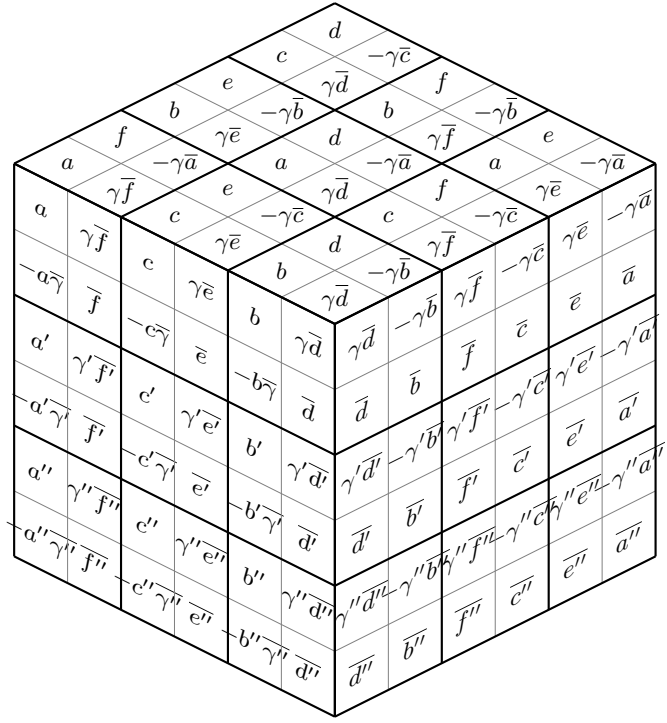}

\caption{Figure demonstrating the block decomposition into $3 \times 3 \times 3$ mini-cubes of shape $2 \times 2 \times 2$. }
\label{fig_cube}
\end{figure}

\begin{observation}\label{empiricallyzauner}
Up to numerical tolerance, each of our numerical MUB-triplets $\{X_1, X_2, X_3\}$ is permutationally unitary equivalent to an MUB-triplet obtained by the construction \eqref{hff} above. In terms of statistics, for approximately 50$\%$ of the random initializations the optimization process
does not reach loss 0, and hence we do not find a numerical MUB-triplet. For about 43$\%$ of the cases the optimization process converges to a numerical MUB-triplet which gives rise to a {\it generic} cube $C$. Finally, for about 7$\%$ of the cases the optimization process converges to a numberical MUB-triplet which gives rise to the {\it exceptional} cube $C$. 
\end{observation}

Upon this numerical evidence, we make the following conjecture. 

\begin{conjecture}\label{mub3}
Any MUB-triplet $(X_1, X_2, X_3)$ in dimension $d=6$ is permutationally unitary equivalent to an MUB-triplet obtained by the construction \eqref{hff}. In particular, there are two essentially different cases:

-- the 2-parameter family of MUB-triplets, described in \cite{Sz}, where the transition matrices $\sqrt{6}X_1X_2^\ast, \sqrt{6}X_2X_3^\ast, \sqrt{6}X_3X_1^\ast$ have the property that one belongs to the Sz\"oll\H osi family $X(\alpha)$, one to the Fourier family $F(x, y)$, and one to the transposed family $F^T(x, y)$

-- an isolated MUB-triplet where all three transition matrices are equivalent to $F_6(1,1)$.

\end{conjecture}

The validity of this conjecture would immediately imply that the maximum number of MUBs in dimension 6 is three. Indeed, by the statement of the conjecture, in any MUB-triplet at least one of the transition matrices $X_jX^\ast_k$ is an element of the Fourier family $F(x,y)$, and an earlier result in the literature \cite{arxiv} ensures that such a matrix cannot be part of a quadruple of MUBs.

\medskip

Next we turn to some non-trivial properties of the Hadamard cubes corresponding to the construction \eqref{eq:zauner}. To this end, we need to remind the reader of some terminology introduced in \cite{mubmols}

\medskip

For a unimodular vector $v=(v_1, \dots, v_d)\in \TT^d$ and an integer vector $\gamma=(\gamma_1, \dots, \gamma_d)\in \Z^d$ we will use the notation $v^\gamma=\prod_{j=1}^d v_j^{\gamma_j}$. 

\medskip

Let $H$ be a complex Hadamard matrix of order $d$, with columns $h_1, \dots, h_d$. For an integer vector $\gamma\in \Z^d$ we define

\begin{equation}\label{gG}
g_H(\gamma)=\sum_{k=1}^d h_k^\gamma, \ \ \ G_H(\gamma)=|g_H(\gamma)|^2=  \left (\sum_{k=1}^d h_k^\gamma \right )\left (\sum_{k=1}^d h_k^{-\gamma}\right )
\end{equation}

With this terminology at hand we can formulate the second main conjecture of this paper. 

\medskip

\begin{conjecture}\label{gg}
Let $(X_1, X_2, X_3)$ be an MUB-triplet in dimension 6, and let $H_1=\sqrt{6}X_1^\ast X_2$, $H_2=\sqrt{6}X_2^\ast X_3$ and $H_3=\sqrt{6}X_3^\ast X_1$ denote the corresponding transition matrices. Let $\pi\in S_6$ be any permutation. Then $G_{H_j}(\pi (3,3,3,-3,-3,-3))G_{H^\ast_j}(\pi (3,3,3,-3,-3,-3))=0$, and $G_{H_j}(\pi(1,1,1,-1,-1,-1))=0$ for all $j=1,2,3$.

Equivalently, in terms of the Hadamard cube $C$ associated with the MUB-triplet $(X_1, X_2, X_3)$, we have $G_{H}(\pi (3,3,3,-3,-3,-3))G_{H^\ast}(\pi (3,3,3,-3,-3,-3))=0$, and $G_H(\pi(1,1,1,-1,-1,-1))=0$ for any slice $H$ of the cube $C$.
\end{conjecture}

\medskip

We can formulate this conjecture as a single algebraic identity as follows. Consider the function
\begin{equation}\label{Gbar}
\tilde{G}_H(\gamma)=\sum_\pi G_H(\pi (\gamma))
\end{equation}
where the summation is taken over all permutations $\pi\in S_6$. Recall that $G_H(\gamma)\ge 0$ for all $\gamma$, so $\tilde{G}_H(\gamma)=0$ happens if and only if each term in the summation is 0. Therefore, $\sum_{j=1}^3\tilde{G}_{H_j}(1,1,1,-1,-1,-1)=0$ is equivalent to $G_{H_j}(\pi(1,1,1,-1,-1,-1))=0$ for each $j$ and each $\pi$, which is the first statement of the conjecture.

\medskip

Similarly, for  $\gamma_2=(3,3,3,-3,-3,-3)$ the equation $\tilde{G}_{H_1}(\gamma_2)\tilde{G}_{H^\ast_1}(\gamma_2)+
\tilde{G}_{H_2}(\gamma_2)\tilde{G}_{H^\ast_2}(\gamma_2)+\tilde{G}_{H_3}(\gamma_2)\tilde{G}_{H^\ast_3}(\gamma_2)=0$ is equivalent to the second statement of the conjecture. 

\medskip

In summary, Conjecture \ref{gg} is equivalent to the following algebraic identity: 

\begin{equation}\label{eq:conj}
\tilde{G}_{H_1}(\gamma_1)+\tilde{G}_{H_2}(\gamma_1)+\tilde{G}_{H_3}(\gamma_1)+\tilde{G}_{H_1}(\gamma_2)\tilde{G}_{H^\ast_1}(\gamma_2)+
\tilde{G}_{H_2}(\gamma_2)\tilde{G}_{H^\ast_2}(\gamma_2)+\tilde{G}_{H_3}(\gamma_2)\tilde{G}_{H^\ast_3}(\gamma_2)=0,
\end{equation}
where $\gamma_1=(1,1,1,-1-1,-1)$ and $\gamma_2=(3,3,3,-3,-3,-3)$. 

\medskip

Note also that, using the identity $\overline{z}=1/z$ for complex numbers of unit length, the function ${G}_H(\gamma)$ can be defined as a rational function of the variables, as indicated in the last equation in \eqref{gG}. As such, equation \eqref{eq:conj} reduces to a single rational function of the variables being 0. In principle, such a relation may be established via Gr\"obner basis techniques, as a consequence of the orthogonality and unbiased relations between the vectors of the bases $X_1, X_2, X_3$.  In particular, equation \eqref{eq:conj} might be implied by conditions a), b) and c) defining inverse orthogonal cubes in Definition \ref{io}.

\medskip

We now show how 
Conjecture \ref{gg} can be an intermediate step in proving Zauner's conjecture that the maximal number of MUBs in dimension 6 is three. To this end, we need to formulate an additional conjecture. (In a previous version of this manuscript it was referred to as a proven fact from the literature, but unfortunately the proof in \cite{chen} turned out to be wrong, as pointed out by \cite{weig}. As such, we need to formulate it as a conjecture here.)

\begin{conjecture}\label{szol}
Any matrix $S\in X(\alpha)$ in the Sz\"oll\H osi family, defined in \ref{eqsz}, cannot be part of an MUB-quadruplet. 
\end{conjecture}

In principle, this conjecture can be proven with the discretization technique in \cite{arxiv}, although the actual implementation could be tedious. As such, we do not consider this conjecture as a major obstacle in the future. Conjecture \ref{gg} (or, equivalently, the algebraic identity \eqref{eq:conj}) should be regarded as the main conjecture in this paper.

\medskip

After this preparation, we can now show how these conjectures would imply that the maximum number of MUBs in dimension 6 is three.

\begin{proposition}\label{propcon1}
The validity of Conjectures \ref{gg} and \ref{szol} implies that the maximal number of MUBs in $\Co^6$ is three.  
\end{proposition} 

\begin{proof}
We first show that 
if for a complex Hadamard matrix $H$ of order 6 the equations 
\begin{equation}\label{13}
G_{H}(\pi(1,1,1,-1,-1,-1))=G_{H}(\pi(3,3,3,-3,-3,-3))=0
\end{equation}
hold for all permutations $\pi$, then $H$ must belong to the Fourier family $F(x,y)$ or to the Sz\"oll\H osi family $X(\alpha)$. Once this is established, we can fall back on the result of \cite{arxiv}, that matrices from the Fourier family $F(x,y)$  cannot be part of any quadruplet of MUBs, and hence Zauner's conjecture follows if $H$ is belongs to the Fourier family. If $H$  belongs to the Sz\"oll\H osi family $X(\alpha)$, we fall back on Conjecture \ref{szol} above.

\medskip

Note that
$G_{H}(\pi(1,1,1,-1,-1,-1))=0$ and $G_{H}(\pi(3,3,3,-3,-3,-3))=0$ for all permutations $\pi$ occurs if and only if 
\begin{equation}\label{Gbar0}
\tilde{G}_H(1,1,1,-1,-1,-1)=0, \ \ \ \tilde{G}_H(3,3,3,-3,-3, -3)=0.
\end{equation}

It remains to show that in this case $H$ belongs to the Fourier family $F(x,y)$ or the Sz\"oll\H osi family $X(\alpha)$. We will need some auxiliary definitions and lemmas during the proof. We first recall a general definition from \cite{BBS} which makes sense in any even dimensions.

\medskip

\begin{definition}\label{def:cancellation}
	A vector $v\in \C^{2n}$ is called \emph{binary}, if it is of the form $(x_1,-x_1\stb x_n,-x_n)$, up to a permutation of the indices. We will call a scalar product $\bra v,w\ket$ \emph{binary}, if the vector $\conj{v}.w$ is binary, where $.$ denotes the element-wise product.
\end{definition}

We can characterize binary vectors by the following lemma. 

\medskip

\begin{lemma}\label{lemma:cancellation}
	A set of complex numbers $\{\al_i\in \C: i=1\stb 2n\}$ is binary if and only if 
\begin{equation}\label{ak}
\sum_{i=1}^{2n} \al_i^k=0
\end{equation}
holds for all odd $k$ in the range $1\leq k \leq 2n-1$. If all $\al_i\in \C$ have unit length, then it is enough to assume \eqref{ak} for all odd $k$ in the range  $1\le k\le n$ \end{lemma}
\begin{proof}
One direction of the statement is trivial. For the other direction, write $p_j=\sum_{j=1}^{2n} \alpha_i^j$ for the $j$th power sum, and $e_j$ for the $j$th elementary symmetric polynomial of the variables $\al_i$. 
Assume that all $p_k=0$ for all odd $k$ in the range $1\le k\le 2n-1$. 
Let \[f(z):=\prod_{j=1}^{2n}(z-\al_j)=\sum_{j=0}^{2n} (-1)^je_jz^{2n-j}\]
It is enough to show that $f$ is an even function, $f(z)=f(-z)$, because in that case $f(z)=g(z^2)$ for some polynomial $g(z)$, and hence the roots of $f$ come in pairs $\pm \alpha_j$. To this end, it is enough to show that all $e_{2j-1}=0$ for $1\le j\le n$. This, in turn, follows from $p_1=e_1=0$, and  induction on $j$ via Newton's identities,
$(2j-1)e_{2j-1}=\sum_{k=1}^{2j-1} (-1)^{k+1} e_{2j-1-k}p_k$. Each term here contains an odd index $p_k$ or an odd index $e_i$ with $i<2j-1$, and hence vanishes by induction. 

 \medskip
 
If all $\al_j$ satisfy $\al_j\conj{\al_j}=1$, then $\conj{e_j}\cdot e_{2n}=e_{2n-j}$, and hence the vanishing of $e_k$ implies the vanishing of $e_{2n-k}$. Therefore, it is enough to assume $p_k=0$ for odd $k$ in the range $1\le k\le n$, and the same proof applies verbatim.
\end{proof}

In dimension 6, unimodular binary vectors have a particularly simple characterization:
\begin{lemma}\label{lemma:cancellingvector}
	The following are equivalent for a zero-sum vector $\vb=(v_1, \dots, v_6) \in \TT^{6}$:
	\begin{itemize}
		\item[1)]  $\vb$ is binary,
		\item[2)] for any $i$, the normalized vector $\vb/v_i$ contains a -1,
		\item[3)] there exists some pair of indices $i,j$, such that $v_i+v_j=0$.
	\end{itemize}
\end{lemma}
\begin{proof}
	Directions $1\then 2$ and $2\then 3$ are trivial. For the implication  $3\then 1$ observe that if $v_i+v_j=0$, then the rest of the entries also sum up to zero, and we can apply Lemma \ref{lemma:cancellation} with $n=2$. 
\end{proof}

\medskip

We will need yet another algebraic condition which implies that a unimodular vector is binary. 

\begin{proposition}\label{alg}
	Assume that $\xb=(x_1\stb x_6)$ is a unimodular vector such that $x_1+\dots +x_6=0$ and  $x_1x_3x_5+x_2x_4x_6=0$. Then $\xb$ is binary, and $(x_2,x_4,x_6)$ can be reordered so that $x_1+x_2=x_3+x_4=x_5+x_6=0$ holds.
\end{proposition}
\begin{proof}
By the assumptions $x_1+\dots +x_6=0, \  x_1x_3x_5+x_2x_4x_6=0$,
	and the fact that each $x_i$ has unit length, we obtain
	\begin{equation}\label{eq:e2}
        x_1x_3+x_3x_5+x_5x_1=(\overline{x_1}+\overline{x_3}+\overline{x_5})x_1x_3x_5=(\overline{x_2}+\overline{x_4}+\overline{x_6})x_2x_4x_6=x_2x_4+x_4x_6+x_6x_2.
	\end{equation}

    Using this fact and the assumptions again, we get

	\begin{align}
		\label{eq:1}
		(x_1+x_2)(x_1+x_4)(x_1+x_6)=\\
x_1^3+x_1^2(x_2+x_4+x_6)+x_1(x_2x_4+x_4x_6+x_6x_2)+x_2x_4x_6=\\
x_1(x_1^2-x_1(x_1+x_3+x_5)+(x_1x_3+x_3x_5+x_5x_1)-x_3x_5)=0,
	\end{align}
where the last equation is trivial, as the expression in the bracket is formally zero.

\medskip

By permutational symmetry, we can assume that the first factor of the product in \eqref{eq:1} is zero, i.e. $x_1+x_2=0$. Then, by Lemma  \ref{lemma:cancellingvector}, the vector $\xb$ must be binary, and hence the negative of $x_3$ must appear among $x_4, x_5, x_6$. If $x_3+x_4=0$ or $x_3+x_6=0$, then the proof is finished.

\medskip

The only remaining case to consider is when $x_3+x_5=0$, in which case $x_4+x_6=0$ also holds. In this case,
	\[
	0=x_1x_3x_5+x_2x_4x_6=-(x_1x_3^2+x_2x_4^2)=x_2(x_3^2-x_4^2),
	\]
	and hence $x_3=\pm x_4$. 
    If $x_3=-x_4$, the proof is finished. Finally, if $x_3=x_4$, then $x_3=-x_6$ due to the relation $x_4+x_6=0$. 
    \end{proof}

In the sequel, we will use the shorthand notation $I\in \binom{n}{k}$ for a $k$-element subset $I$ of $\{1\stb n\}$. Given $I\in\binom{n}{k}$, let $\mu(I)=(\mu_1\stb \mu_n)\in \{\pm1\}^n$ be defined by
\begin{equation}\label{eq:muI}
	\mu_j=\begin{cases}
		-1,\qquad & j\in  I\\
		1,\qquad &j\not\in I.
	\end{cases}
\end{equation}

\begin{corollary}\label{cor:RHtoscalar}
	Let $\av, \bv$ be unimodular vectors in $\CC^6$, and assume that for some $I\in\binom{6}{3}$ and $\la\in \TT$ we have
	\[\av^{\mu(I)}=\la, \ \ \ \bv^{\mu(I)}=-\la,
    \]
	where we use the notation $\xb^\ga=\prod x_i^{\ga_i}$.
	
	Then the scalar product $\bra a, b\ket$ is binary,  such that the cancellations are compatible with $I$ in the following sense: there is a bijection $\varphi:I\to I^c$, such that for all $i\in I$:
	\[
	a_i/b_{i}+a_{\varphi(i)}/b_{\varphi(i)}=0
	\]
\end{corollary}
\begin{proof}
By permutational symmetry, we can assume without loss of generality that $I=\{1,3,5\}$, and then we can apply Proposition \ref{alg} with   $x_i=a_i/b_i$.
\end{proof}

We remind the reader that given a complex Hadamard matrix $H$, the operation "{\it dephasing $H$ by row $i$ and column $j$}" means that we multiply the rows and columns of $H$ by unit numbers in such a way that row $i$ and column $j$ become the constant 1 vectors. We can give a simple application of  Corollary \ref{cor:RHtoscalar} in terms of this dephasing operation. 

\begin{proposition}\label{prop:row_minus_one}
	Let $I\in \binom{6}{3}$, and let $H$ be a complex Hadamard matrix of order 6. Assume that $H$ is $I$-binary, i.e for the columns $h_1, \dots, h_6$ of $H$, we have that the vector  $(h_1^{\mu(I)}, \dots, h_6^{\mu(I)}$ is binary. Then, after dephasing $H$ by a row in $I$ and an arbitrary column, there will be a row in $I^c$ which contains a $-1$ entry in the dephased matrix. 
\end{proposition}
\begin{proof}
For simplicity, introduce the notation $\alpha_j=h_j^{\mu(I)}$. 
Pick an arbitrary column of $H$. By permutational symmetry, we can assume it is $h_1$ .
	Since $(\al_1\stb \al_6)$ is binary, there exists $j$ such that $\al_1+\al_j=0$. Without loss of generality, assume that $\al_1+\al_2=0$. This means that $h_1^{\mu(I)}=\al_1$ and  $h_2^{\mu(I)}=-\al_1$. Let us dephase $H$ by $h_1$ and a row $i$ in $I$. By applying Corollary \ref{cor:RHtoscalar} we get a bijection $\varphi:I\to I^c$,  and \[h_{1,i}/h_{2,i}+h_{1,\varphi(i)}/h_{2,\varphi(i)}=1/1+1/h_{2,\varphi(i)}=0\] which is exactly the claim.
\end{proof}

\begin{proposition}\label{tm1}
	Let $H$ be a Hadamard matrix which satisfies $g_H(\pi(1,1,1,-1,-1,-1))=0$, and
	$g_H(\pi(3,3,3,-3,-3,-3))=0$ for all permutations $\pi$. Then after dephasing $H$ by an arbitrary row and column, there exist three distinct rows containing a $-1$ entry.
\end{proposition}
\begin{proof}
	First, recall that if for $I\in\binom{6}{3}$, $g_H(\mu(I))=g_H(3\cdot \mu(I))=0$, then $H$ is $I$-binary by Lemma \ref{lemma:cancellation}, so Proposition \ref{prop:row_minus_one} can be applied.
	
	Pick an arbitrary row and column, say the first one, and dephase the matrix. Then taking $I=(1,5,6)$, by Proposition \ref{prop:row_minus_one}, there exists a row in $I^c$, which contains a $-1$. Without loss of generality, let this be column number $2$. Then taking $I=(1,2,6)$, we find a second row in $I^c$ containing a $-1$, let this be row number 3. Finally, taking $I$ to be $(1,2,3)$, its complement must also contain a $-1$, and hence we find the third row with a $-1$ entry. Thus the matrix contains three $-1$'s in different rows, as claimed.
\end{proof}

\medskip

Finally, we can invoke the results of \cite{ms}, stating that if a complex Hadamard matrix $H'$ of order 6 contains -1 entries in three distinct columns, then $H'$ must belong to the transposed Fourier family $F^T(x,y)$ or the Sz\"oll\H osi family $X(\al)$. As we have -1 entries in three distinct rows of $H$, we conclude that $H$ belongs to the Fourier family $F(x,y)$ or the Sz\"oll\H osi family $X(\al)$. 

\medskip 

This completes the proof of Proposition \ref{propcon1}.
\end{proof}

In summary, we have exhibited a single algebraic identity \eqref{eq:conj} for MUB-triplets which would imply (given the validity of the minor Conjecture \ref{szol}) that the maximum number of MUBs in dimension 6 is three. Such an algebraic identity could, in principle, be proven using the basic cube properties (i-iv) in Proposition \ref{hcube}.  

\section{Acknowledgements}

The authors are indebted to Ferenc Sz\"oll\H osi for insightful remarks. 

M. M. was supported by grants NKFIH-132097 and NKFIH-146387. D. V. and Á.K.M. were  supported by the
Ministry of Innovation and Technology NRDI Office within the framework of the Artificial
Intelligence National Laboratory (RRF-2.3.1-21-2022-00004). Á.K.M. was also supported by the Hungarian NRDI Office grants NKFIH K-138828 and PD-145995. 
M.W. was supported by the National Research, Development and Innovation Office of Hungary (NKFIH) via the research grants K 146380 and EXCELLENCE 151342, and by the Ministry of Culture and Innovation and the National Research, Development and Innovation Office within the Quantum Information National Laboratory of Hungary (Grant No. 2022-2.1.1-NL-2022-00004).

\end{document}